\documentclass[11pt]{amsart}
\usepackage{mathptmx}
\usepackage{amscd,amssymb,amsmath,amsthm,amsfonts}

\newtheorem{thm}{Theorem}[section]
\newtheorem{prop}[thm]{Proposition}
\newtheorem{lemma}[thm]{Lemma}
\newtheorem{cor}[thm]{Corollary}

\theoremstyle{definition}
\newtheorem{defn}{Definition}[section]

\numberwithin{equation}{section}

\DeclareMathAlphabet{\matheur}{U}{eur}{m}{n}
\DeclareMathAlphabet{\matheus}{U}{eus}{m}{n}
\DeclareMathAlphabet{\matheuf}{U}{euf}{m}{n}

\newcommand{\abs}[1]{\left\lvert#1\right\rvert}

\newcommand{\R}{{\mathbb R}}
\newcommand{\C}{{\mathbb C}}

\newcommand{\cO}{{\mathcal O}}

\author{Vestislav Apostolov}
\address{D\'epartement de math\'ematiques, Universit\'e du Qu\'ebec \`a Montr\'eal, Case postale 8888, succursale centre-ville Montr\'eal (Qu\'ebec) H3C 3P8, Canada}
\email{\tt  apostolov.vestislav@uqam.ca}

\author{Dmitry Jakobson}
\address{Department of Mathematics and Statistics, McGill University, 805 Sherbrooke Street West,
Montr\'eal (Qu\'ebec) H3A0B9, Canada}
\email{\tt  jakobson@math.mcgill.ca }

\author{Gerasim Kokarev}
\address{Mathematisches Institut der Universit\"at M\"unchen, Theresienstr. 39, D-80333 M\"unchen, Germany}
\email{\tt Gerasim.Kokarev@math.lmu.de}
\subjclass[2010]{49R05,35P99,53C55,32Q20} \keywords{Laplacian eigenvalues, extremal metric, K\"ahler manifold, K\"ahler-Einstein manifold.}

\title{An extremal eigenvalue problem in K\"ahler geometry}

\begin{document}

\dedicatory{Dedicated to Professor Paul Gauduchon on the occasion of his 70'th birthday.}

\begin{abstract} 
We study Laplace eigenvalues $\lambda_k$ on K\"ahler manifolds as functionals on the space of K\"ahler metrics with cohomologous K\"ahler forms. We introduce a natural notion of a $\lambda_k$-extremal K\"ahler metric and obtain necessary and sufficient conditions for it. A particular attention is paid to the $\lambda_1$-extremal properties of K\"ahler-Einstein metrics of positive scalar curvature on manifolds with non-trivial holomorphic vector fields.
\end{abstract}

\maketitle


\section{Introduction}
\subsection{Motivation}
Let $M$ be a closed manifold. For a Riemannian metric $g$ on $M$ we denote by
$$
0=\lambda_0(g)<\lambda_1(g)\leqslant\lambda_2(g)\leqslant\ldots\leqslant\lambda_k(g)\leqslant\ldots
$$
the eigenvalues of the Laplace-Beltrami operator $\Delta_g$ repeated according to their multiplicity. In real dimension $2$, by classical results of Hersch~\cite{He70}, Yang and Yau~\cite{YY}, and Li and Yau~\cite{LY}, the first eigenvalue $\lambda_1(g)$ is bounded when the Riemannian metric $g$ ranges over metrics of fixed volume. A basic question is: for a given conformal class $c$ on $M$, is there a metric that maximizes $\lambda_1(g)$ among metrics $g\in c$ with ${\rm vol}(M,g)= 1$? What are its properties? When $M$ is a sphere or a projective plane, the answers go back to the classical results of Hersch~\cite{He70} and Li and Yau~\cite{LY}. For higher genus surfaces this circle of questions have been studied extensively in the last decades, see~\cite{El00,El03, EGJ, JLNNP, JNP06, Na96} and the most recent papers~\cite{Ko,Na10,Pet}. 

In particular, Nadirashvili and Sire~\cite{Na10}, developing earlier ideas by Nadirashvili~\cite{Na96}, have stated an existence theorem for $\lambda_1$-maximizers in conformal classes along with an outlined proof. This  statement has been improved by Petrides~\cite{Pet}, who has also given a rigorous argument for it, using previous work~\cite{FS2,Ko}. Mention that the $\lambda_1$-maximizers given by the above existence theorems may have conical singularities, and can be described as metrics that admit harmonic maps into round spheres by their first eigenfunctions. The latter statement actually holds for arbitrary $\lambda_1$-maximizers (and even $\lambda_k$-extremals for any $k\geqslant 1$) with conical singularities, see~\cite{Ko}, where many more general statements in this direction have been proven. Besides, the combination of~\cite[Theorem~1]{Pet} and~\cite[Theorem~$E_1$]{Ko} shows that the set of all conformal $C^{\infty}$-metrics with conical singularities that maximize $\lambda_1$ is compact, see also~\cite{Ko10}.

\subsection{Eigenvalue problems on K\"ahler manifolds}
The purpose of this paper is to describe an extremal eigenvalue problem in higher dimensions on a K\"ahler manifold $(M,J,g,\omega)$, which generalizes the above extremal problem on Riemannian surfaces. Here we view the eigenvalue $\lambda_k(g)$ as a functional on the space $\mathcal K_\Omega(M,J)$ that is formed by K\"ahler metrics $g$ whose K\"ahler forms $\omega$ represent a given de Rham cohomology class $\Omega$. When $(M,J,g)$ is a Riemannian surface with a volume form $\omega$, the space $\mathcal K_{[\omega]}(M,J)$ is precisely the set of metrics of a fixed area that are conformal to $g$, see the discussion in Sect.~\ref{s:extremal-general}. By classical work due to Bourguignon, Li, and Yau~\cite{BLY}, the first eigenvalue $\lambda_1(g)$ is bounded on $\mathcal K_\Omega(M,J)$ when $(M,J)$ is projective and the de Rham class $\Omega$ belongs to $H^2(M,\mathbb Q)$. The eigenvalue bound in~\cite{BLY} shows that the Fubini-Study metric on ${\mathbb C} P^m$ is a $\lambda_1$-maximizer in its K\"ahler class. Recently, these results have been generalized by Arezzo, Ghigi, and Loi~\cite{AGL} to more general K\"ahler manifolds that admit holomorphic stable vector bundles over $M$ with sufficiently many sections. In particular, they show that the symmetric K\"ahler-Einstein metrics on the Grassmannian spaces are also $\lambda_1$-maximizers in their K\"ahler classes. Moreover, as is shown in~\cite{BG}, so are symmetric K\"ahler-Einstein metrics on Hermitian symmetric spaces of compact type. These results motivate a further investigation of spectral geometry of K\"ahler-Einstein metrics of positive scalar curvature; see~\cite{CDS,T} for the general existence theory of such metrics.

Following the ideas in~\cite{Na96,SI} for other extremal eigenvalue problems, we introduce the notion of a $\lambda_k$-{\em extremal K\"ahler metric} under the deformations in its K\"ahler class. The class of such extremal metrics contains $\lambda_k$-maximizers in $\mathcal K_\Omega(M,J)$. We assume that the metrics under consideration are always smooth, and show that for a $\lambda_k$-extremal K\"ahler metric there exists a collection of non-trivial $\lambda_k$-eigenfunctions $f_1,\ldots,f_\ell$ such that
$$
\lambda_k^2 \Big(\sum_{i=1}^{\ell} f_i^2 \Big) - 2\lambda_k \Big(\sum_{i=1}^{\ell}\abs{\nabla f_i}^2\Big)  + \sum_{i=1}^{\ell}\abs{dd^c f_i}^2=0. 
$$
For the first eigenvalue the latter hypothesis is also sufficient for a metric to be $\lambda_1$-extremal. This statement shows that for a $\lambda_k$-extremal K\"ahler metric $g$ the eigenvalue $\lambda_k(g)$ is always multiple (Corollary~\ref{c:multiple}), and allows to produce examples of $\lambda_1$-extremal metrics as products.

We proceed with considering in more detail the case when $(M,J)$ is a complex Fano manifold with non-trivial holomorphic vector fields. Recall that, by a classical result of Matsushima~\cite{Ma57}, for a K\"ahler-Einstein metric $g$ on $(M,J)$ the $\lambda_1$-eigenfunctions are potentials of Killing vector fields. Using this fact, we show that a K\"ahler-Einstein metric $g$ is $\lambda_1$-extremal if and only if there exist non-trivial Killing potentials $f_1,\ldots,f_\ell$ such that the function $\sum f_i^2$ is also a Killing potential. As an application, we conclude that a toric Fano K\"ahler-Einstein manifold whose connected group of automorphisms is a torus is not $\lambda_1$-extremal. For example, so is a K\"ahler-Einstein metric on ${\mathbb C}P^2$ blown up at three points in a generic position, see~\cite{siu,T0}. As another example, we show that K\"ahler-Einstein manifolds different from ${\mathbb C}P^m$  that admit hamiltonian $2$-forms of order $\geqslant 1$ are also never $\lambda_1$-extremal. This conclusion applies to the non-homogeneous K\"ahler-Einstein metrics from \cite{Koi90,KS86,KS88}.


\section{An extremal eigenvalue problem}
\label{s:extremal-general}
\subsection{Statement of the problem}
Let $(M,g,J,\omega)$ be a compact K\"ahler manifold of real dimension $n=2m$. Recall that due to the $\partial\bar\partial$-lemma any K\"ahler metric $\tilde g$ whose K\"ahler form $\tilde\omega(\cdot,\cdot)=\tilde g(J\cdot,\cdot)$ is co-homologous to $\omega$ has the form $\omega+dd^c\varphi$, where $d^c=JdJ^{-1}=i(\partial-\bar\partial)$ and the action of $J$ on the cotangent bundle is defined via the duality with respect to $g$. The smooth function $\varphi$ above is determined uniquely by the condition
$$
\int_M\varphi v_g=0,$$
where  $v_g=\omega^m/m!$
is the Riemannian volume form on $M$.  By $\mathcal K_\Omega(M,J)$  we denote the space of K\"ahler metrics on $(M,J)$ whose K\"ahler forms represent a given de Rham cohomology class $\Omega$. For a representative K\"ahler form $\omega$ it can be identified with the space of functions
$$
\{\varphi\in C^\infty(M): \omega+dd^c\varphi>0,\int_M\varphi v_g=0\}.
$$
Besides, for any function $\varphi$ with a zero mean-value there exists $\varepsilon>0$ such that $\omega+tdd^c\varphi\in\mathcal K_\Omega(M,J)$ for all $\abs{t}<\varepsilon$, and the space $C^\infty_0(M)$, formed by such functions $\varphi$, can be thought as the tangent space at $g\in\mathcal K_\Omega(M,J)$.

For a K\"ahler metric $g$ on $(M,J)$ we denote by
$$
0=\lambda_0(g)<\lambda_1(g)\leqslant\lambda_2(g)\leqslant\ldots\leqslant\lambda_k(g)\leqslant\ldots
$$
the eigenvalues of the Laplace--Beltrami operator $\Delta_g=\delta d$ acting on functions, where $\delta$ is the $L_2$-adjoint of the exterior derivative $d$ with respect to $g$. Recall that due to the standard Kato--Rellich perturbation theory the functions $t\mapsto\lambda_k(g_t)$ have left and right derivatives for analytic deformations $g_t$. We view the eigenvalues $\lambda_k(g)$ as functionals on the space $\mathcal K_\Omega(M,J)$ and introduce the following definition.
\begin{defn}
\label{ex:defn}
A K\"ahler metric $g\in\mathcal K_\Omega(M,J)$ is called $\lambda_k$-{\it extremal}, if for any analytical deformation $g_t\in\mathcal K_\Omega(M,J)$ with $g_0=g$ the following relation holds
\begin{equation}
\label{extremal:def}
\left.\frac{d}{dt}\right|_{t=0-}\!\!\lambda_k(g_t)\cdot\left.\frac{d}{dt}\right|_{t=0+}\!\!\lambda_k(g_t)\leqslant 0.
\end{equation}
\end{defn}

It is straightforward to see that a K\"ahler metric is $\lambda_k$-extremal if and only if either the inequality
$$
\lambda_k(g_t)\leqslant\lambda_k(g)+o(t)\qquad\text{as }t\to 0,
$$
or the inequality
$$
\lambda_k(g_t)\geqslant\lambda_k(g)+o(t)\qquad\text{as }t\to 0
$$
occurs. In particular, we see that any $\lambda_k$-maximizer in $\mathcal K_\Omega(M,J)$ is $\lambda_k$-extremal. The following two remarks are consequences of the results in~\cite{SI}. First, for the first eigenvalue only the first of the above inequalities may occur. Second, for a deformation $\omega_t=\omega+dd^c\varphi_t$ the validity or the failure of relation~\eqref{extremal:def} depends only on the function $\dot\varphi=(d/dt)|_{t=0}\varphi_t$, and hence, in Definition~\ref{ex:defn} we may consider only deformations with $\varphi_t=t\varphi$, where $\varphi\in C^\infty_0(M)$.

In the sequel we use the fourth order differential operator $L(f)$ defined as $\delta^c\delta(fdd^cf)$, where $\delta$ and $\delta^c$ stand for the $L_2$-adjoints of $d$ and $d^c$ respectively. Recall that they satisfy the relations
$$
\delta \psi = -\sum_{i=1}^{2m} \imath_{e_i} (D_{e_i}\psi)\quad\text{and}\quad \delta^c \psi = -\sum_{i=1}^{2m} \imath_{Je_i} (D_{e_i}\psi),
$$
where $D$ is the Levi-Civita connection of $g$, and $\{e_1,Je_1,\ldots,e_m,Je_m\}$ is a $J$-adapted orthonormal frame. We then calculate
\begin{multline}
\label{L}
L(f)=\delta^c \Big(-(dd^c f) (df^{\sharp}, \cdot)+ f d^c \delta d f \Big)\\=\sum_{i=1}^{2m}\imath_{J_{e_i}}(D_{e_i}dd^cf )(df^{\sharp}, \cdot)-\sum_{i=1}^{2m} (\imath_{Je_i}(dd^cf) (D_{e_i} df{\sharp}), \cdot)+f \delta^cd^c \delta df - (df, d \Delta f) \\
=(\delta^c dd^cf, df)+(dd^cf, dd^cf)+ f \Delta_g^2 f-(\Delta df, df)  \\
=(dd^cf, dd^cf)  + f \Delta_g^2 f - 2(\Delta df, df), 
\end{multline}       
where we used the facts that $dd^cf$ is $J$-invariant and that  $J$ commutes with $D$ and $\Delta_g$ (when acting on $1$-forms), as well as the standard K\"ahler identities between the operators $d$, $d^c$, $\delta$, $\delta^c$, and $\Delta_g$.


The proof of the following statement is close in the spirit to the arguments in~\cite{Na96,SI} and is given at the end of the section.
\begin{thm} \label{l:extremal}
Let $g\in\mathcal K_\Omega(M,J)$ be a $\lambda_k$-extremal K\"ahler metric. Then there exists a finite collection $f_1,\ldots,f_\ell$ of non-trivial eigenfunctions corresponding to $\lambda_k(g)$ such that
\begin{equation}
\label{extremal}
\sum_{i=1}^{\ell}  L(f_i)=\lambda_k^2 \Big(\sum_{i=1}^{\ell} f_i^2 \Big) - 2\lambda_k \Big(\sum_{i=1}^{\ell}\abs{\nabla f_i}^2\Big)  + \sum_{i=1}^{\ell}\abs{dd^c f_i}^2=0.   
\end{equation}
For $k=1$ the existence of such a collection of eigenfunctions is sufficient for a K\"ahler metric $g$ to be $\lambda_1$-extremal.
\end{thm}
Denote by $E_k=E_k(g)$ the eigenspace formed by eigenfunctions corresponding to the eigenvalue $\lambda_k(g)$.
The operator $L$ restricted to $E_k$ takes the form
$$
L(f)=\lambda_k^2f^2-2\lambda_k\abs{\nabla f}^2+\abs{dd^cf}^2.
$$
Considering maximal and minimal values of a function $f\in E_k$, it is straightforward to see that it has a trivial kernel. Thus, we obtain the following corollary.
\begin{cor}\label{c:multiple}
Let $g\in\mathcal K_\Omega(M,J)$ be a $\lambda_k$-extremal K\"ahler metric. Then the eigenvalue $\lambda_k(g)$ is multiple.
\end{cor}
As another consequence of Theorem~\ref{l:extremal}, we see that the notion of $\lambda_1$-extremality behaves well under products.
\begin{cor}
\label{c:product}
Let $g\in\mathcal K_{\Omega}(M)$ be a $\lambda_1$-extremal K\"ahler metric on $(M,J)$. Then for any K\"ahler metric $g'$ on $(M',J')$ such that $\lambda_1(g')\geqslant\lambda_1(g)$ the product metric $g\times g'$ is $\lambda_1$-extremal along deformations in its K\"ahler class on $(M,J)\times (M',J')$.
\end{cor}
Mention that the hypothesis $\lambda_1(g')\geqslant\lambda_1(g)$ in the corollary above always holds after an appropriate scaling of either of the metrics.

\subsection{Discussion and basic examples}
We proceed with considering the case when $m=1$, that is when $M$ is an oriented Riemann surface. Let $g$ be a Riemannian metric and $\omega=v_g$ be its volume form. It is straightforward to see that for any smooth function $\varphi\in C^\infty(M)$ the hypothesis $\omega+dd^c\varphi>0$ holds if and only if $1-\Delta_g\varphi>0$. Thus, the space $\mathcal K_{[\omega]}(M,J)$ can be defined as
$$
\mathcal K_{[\omega]}(M,J)=\{\varphi\in C^\infty(M):1-\Delta_g\varphi>0,\int_M\varphi v_g=0\}.
$$
The following lemma is elementary; we state it for the convenience of references.
\begin{lemma}
Let $(M,g)$ be a Riemannian surface, and $\omega$ be its volume form. Then the space of K\"ahler metrics  in $\mathcal K_{[\omega]}(M,J)$ coincides with the space of Riemannian metrics $\tilde g$ that are conformal to $g$ and such that ${\rm vol}(M,{\tilde g}) ={\rm vol}(M,g)$.
\end{lemma}
\begin{proof}
Let $\varphi$ be a function from the above space $\mathcal K_{[\omega]}(M,J)$. Then the metric $\tilde g=(1-\Delta_g\varphi)g$ is clearly has the same volume as the metric $g$. Conversely, for a given conformal metric $\tilde g=e^{\sigma}g$ of the same volume as $g$ the equation
$$
e^{\sigma}=1-\Delta_g\varphi
$$
has a unique solution $\varphi\in C^\infty(M)$ with zero mean-value, see~\cite{GT} for standard existence results for solutions of elliptic equations.
\end{proof}
As a consequence of the lemma above, we see that a Riemannian metric $g$ on $M$ is $\lambda_k$-extremal in the sense of Definition~\ref{extremal:def} if and only if it is $\lambda_k$-extremal under volume preserving conformal deformations.  There have been constructed a number of examples of various $\lambda_1$-extremal and $\lambda_1$-maximal metrics in the literature, and as is known~\cite{Na10,Pet}, every conformal class on a closed surface contains a $\lambda_1$-maximizer, which may have conical singularities. Thus, using Corollary~\ref{c:product}, we obtain a variety of examples of $\lambda_1$-extremal K\"ahler metrics by taking the products of $\lambda_1$-extremal Riemannian surfaces and K\"ahler manifolds.

Recall that by~\cite{El03,SI} for a metric $g$ that is $\lambda_k$-extremal under the volume preserving conformal deformations there exists a collection $f_1,\ldots,f_\ell$ of non-trivial $\lambda_k$-eigenfunctions such that $\sum f_i^2=1$. When $m=1$, we see that the latter condition coincides  with the necessary condition given by Theorem~\ref{l:extremal}. Indeed, in this case, for a $\lambda_k$-eigenfuction $f$ the operator $L(f)$ takes the form
$$
L(f)=2(\lambda_k^2f^2-\lambda_k\abs{\nabla f}^2),
$$
where we used the identities $(\Delta_g f)\omega=-dd^cf$ and $\abs{\omega}^2=1$. Now the relation
$$
\Delta_g(\sum f_i^2)=2(\lambda_k \sum f_i^2-\sum\abs{\nabla f_i}^2)=\lambda_k^{-1}\sum L(f_i)
$$
implies the claim. More generally, in higher dimensions $m\geqslant 2$ the hypothesis $\sum L(f_i)=0$ in Theorem~\ref{l:extremal} is equivalent to the relation
$$
\Delta_g(\sum f_i^2)=\lambda_k^{-1}m(m-1)(\sum dd^cf_i\wedge dd^cf_i)\wedge\omega^{m-2}/\omega^m.
$$
Here we used the fact that $(\Delta_g f)^2-\abs{dd^cf}^2=m(m-1)dd^cf\wedge dd^cf\wedge\omega^{m-2}/\omega^m$, see identity~\eqref{quadratic} in the proof of Theorem~\ref{l:extremal} below.

\subsection{Proof of Theorem~\ref{l:extremal}}
Below we use the conventions and basic identities from \cite[Ch.~2]{besse}; in particular,  $|\omega|^2=m$ and $m\cdot dd^c f \wedge \omega^{m-1}= (-\Delta_g f)\omega^m$. We start with the following lemma.
\begin{lemma}
\label{l:condition}
Let $g\in\mathcal K_\Omega(M,J)$ be a $\lambda_k$-extremal K\"ahler metric, and $E_k$ be an eigenspace for $\lambda_k(g)$. Then for any function $\varphi\in C^\infty_0(M)$ the quadratic form
$$
Q_\varphi(f)=\int_M\varphi L(f)v_g
$$
is indefinite on $E_k$. For $k=1$ the hypothesis that the form $Q_\varphi$ is indefinite for any $\varphi\in C^\infty_0(M)$ is also sufficient for a K\"ahler metric $g$ to be $\lambda_k$-extremal.
\end{lemma}
\begin{proof}
For any $\varphi\in C^\infty_0(M)$ consider the K\"ahler deformation $\omega_t=\omega+tdd^c\varphi$, defined for a sufficiently small $\abs{t}$. By~\cite[Theorem~2.1]{SI} for a proof of the lemma it is sufficient to show that the form $Q_\varphi$ satisfies the relation
\begin{equation}
\label{aux1}
Q_\varphi(f)=\int_M f(\dot\Delta_\varphi f)v_g,
\end{equation}
where $\dot\Delta_\varphi f$ stands for the value $(d/dt)|_{t=0}(\Delta_{g_t}f)$. First, we claim that the operator $\dot\Delta_\varphi$ satisfies the identity
\begin{equation}
\label{aux2}
\dot\Delta_\varphi f=(dd^c f, dd^c\varphi).
\end{equation}
Indeed, differentiating the relation
$$
m\cdot dd^c f  \wedge \omega_t^{m-1}=(-\Delta_{g_t} f)\omega_t^m,
$$
we obtain
\begin{equation}
\label{variation}
dd^c f \wedge dd^c\varphi \wedge \omega^{m-2}/(m-2)! = \Big((\Delta_{g} f)(\Delta_{g}\varphi)-\dot\Delta_\varphi f\Big) \omega^m /m!.
\end{equation}
Denote by $\wedge^{1,1}(M)$ the bundle of real $(1,1)$-forms on the manifold $(M,J)$, that is $2$-forms $\psi$ such that $\psi(J\cdot, J\cdot) = \psi(\cdot, \cdot)$. For a given K\"ahler metric $(g, \omega)$ on $(M,J)$, it decomposes as a direct $g$-orthogonal sum
\begin{equation}
\label{U-split}
\wedge^{1,1} (M) = \R \omega \oplus \wedge^{1,1}_0(M)
\end{equation}
of the irreducible $U(m)$-invariant subspaces of $(1,1)$-forms  proportional to $\omega$ and primitive (trace-free) $(1,1)$-forms, respectively.  Let us consider the symmetric $U(m)$-invariant bilinear form on $\wedge^{1,1}(M)$ defined as
$$q(\phi, \psi)=  ({\rm tr}_{\omega}\phi) ({\rm tr}_{\omega} \psi) -  \frac{\phi \wedge \psi \wedge (\omega^{m-2}/(m-2)!)}{\omega^m/m!},$$ 
where 
$${\rm tr}_{\omega} \psi = (\psi, \omega)_g = \frac{\psi\wedge (\omega^{m-1}/(m-1)!)}{\omega^m/m!}.$$
Since the form $q(\cdot,\cdot)$ leaves the two irreducible factors in the decomposition~\eqref{U-split} orthogonal, by the Sch\"ur lemma we conclude that it is proportional to the induced Euclidean product $(\cdot, \cdot)_g$ on each factor in~\eqref{U-split}. Evaluating $q(\omega, \omega)$ and $q(\psi_0, \psi_0)$ with $\psi = \alpha \wedge J\alpha$ for a unitary $1$-form $\alpha$, we see  that in fact the form $q(\cdot, \cdot)$ coincides with $(\cdot, \cdot)_g$, that is
\begin{equation}
\label{quadratic}
({\rm tr}_{\omega}\phi) ({\rm tr}_{\omega} \psi) -  \frac{\phi \wedge \psi \wedge (\omega^{m-2}/(m-2)!)}{\omega^m/m!}=(\phi,\psi)_g.
\end{equation}
Now combining the last relation with~\eqref{variation}, we obtain identity~\eqref{aux2}. Using the latter we have
$$
Q_\varphi(f)=\int_M\varphi L(f)v_g=\int_M f(dd^cf,dd^c\varphi)v_g=\int_M f(\dot\Delta_\varphi f)v_g,
$$
and thus, obtain relation~\eqref{aux1}.
\end{proof}
By Lemma~\ref{l:condition}, in order to prove  the theorem it is sufficient to show that the quadratic form $Q_\varphi(f)$ is indefinite on $E_k$ if and only if there exists a collection of eigenfunctions $f_1,\ldots, f_\ell\in E_k$ such that
\begin{equation}
\label{extremal'}
\sum_{i=1}^\ell L(f_i)=0\qquad\text{and}\qquad\sum_{i=1}^\ell\int_Mf_i^2v_g=1.
\end{equation}
Consider the convex subset
$$
K=\left\{\sum_{i}L(f_i): f_i\in E_k, \sum_{i}\int_Mf_i^2v_g=1\right\}
$$
in the space $L^2(M)$. We are going to show that the form $Q_\varphi(f)$ is indefinite if and only if $0\in K$. Suppose that $Q_\varphi(f)$ is indefinite for any $\varphi\in C^\infty_0(M)$ and $0\notin K$. Then by the Hahn--Banach separation theorem there exists a function $\psi\in L^2(M)$ and $\varepsilon>0$ such that
$$
\int_M\psi uv_g\geqslant\varepsilon>0\quad\text{ for any }u\in K.
$$
Since the set $K$ lies in a finite-dimensional subspace, then choosing $\varepsilon>0$ smaller, if necessary, by approximation we may assume that the function $\psi$ belongs to $C^\infty(M)$. Define $\psi_0$ as the zero mean-value part of $\psi$, that is
$$
\psi_0=\psi-\frac{1}{{\rm vol}(M,g)}\int_M\psi v_g.
$$
Since the operator $L(f)$ takes values among zero mean-value functions, we obtain
$$
Q_{\psi_0}(f)=\int_M\psi_0L(f)v_g=\int_M\psi L(f)v_g>0
$$
for any non-trivial $f\in E_k$. Thus, we arrive at a contradiction with the assumption that the form $Q_{\varphi}$ is indefinite for any $\varphi\in C^\infty_0(M)$.

Conversely, given a collection $f_1,\ldots,f_\ell\in E_k$ that satisfy relationships~\eqref{extremal'}, we have
$$
\sum_{i=1}^\ell Q_\varphi(f_i)=\int_M\varphi\left(\sum_{i=1}^\ell L(f_i)\right)v_g=0.
$$
Thus, the quadratic form $Q_\varphi(f)$ is indeed indefinite for any $\varphi\in C^\infty_0(M)$. \qed

\section{K\"ahler--Einstein manifolds with a non-trivial automorphism group}\label{s:KE}
We now specialize the considerations to the case when $(g, J, \omega)$ is a K\"ahler--Einstein manifold with positive scalar curvature, that is the Ricci form $\rho$ is a positive constant multiple of the K\"ahler form $\omega$. Rescaling the metric, we may assume that $\rho=\omega$, or equivalently,  the K\"ahler class is $\Omega= 2\pi c_1(M,J)$. Under this assumption, the scalar curvature ${\rm Scal}_g$ equals $2m$.

Since the manifold $(M,J)$ is Fano, by~\cite{Kob61,yau} it is simply-connected. In particular, the first de Rham cohomology group vanishes, and hence, any real holomorphic vector field $X$ can be uniquely written in the form
$$ 
X = {\rm grad}_g h_X + J {\rm grad}_g f_X, 
$$
where $h_X$ and $f_X$ are smooth functions with zero mean-values, see~\cite[Cor.~2.126]{besse} for details. The complex-valued function $h_X+if_X$ is called the {\it holomorphy potential} of $X$. For a Killing vector field $Y$, the above decomposition reduces to
$$Y = J {\rm grad}_g f_Y,$$
and the corresponding function $f_Y\in C_{0}^{\infty}(M)$ is called the {\it Killing potential} of $Y$. The following statement is due to  Matsushima~\cite{Ma57}, see also \cite[Ch.~3]{gauduchon-book}.
\begin{prop}
\label{thm:matsushima} 
Let $(M, g, J, \omega)$ be a compact K\"ahler--Einstein manifold with scalar curvature ${\rm Scal}_g= 2m$.  Then the Lie algebra $\mathfrak{h}(M)$ of real holomorphic vector fields on $(M,J)$ decomposes as the direct sum
$$\mathfrak{h}(M) = \mathfrak{k}(M,g) \oplus J \mathfrak{k}(M,g),$$
where $\mathfrak{k}(M,g)$ is the sub-algebra of Killing vector fields for $g$. Moreover, the algebra $\mathfrak{k}(M,g)$ is Lie algebra isomorphic to the space
$$E_{1}(g)=\{f \in C_{0}^{\infty}(M) : \Delta_g (f)=2 f\}$$
equipped with the Poisson bracket of functions with respect to $\omega$, via the map $f\to J {\rm grad}_g f$.
Furthermore, the first eigenvalue satisfies the inequality $\lambda_1(g)\geqslant 2$.
\end{prop}
As a consequence of the proposition above, we see that $\lambda_1(g)=2$ if and only if the connected component of the identity ${\rm Aut}_0(M,J)$ of the group of biholomorphic automorphisms of $(M,J)$ is non-trivial. In this case, we have the following statement.
\begin{thm}\label{thm:main} 
Let $(M,g,J, \omega)$ be a compact K\"ahler--Einstein manifold with scalar curvature ${\rm Scal}_g=2m$, and suppose that the connected component ${\rm Aut}_0(M,J)$ of the automorphism group is non-trivial. 
Then the metric $g$ is $\lambda_1$-extremal in $\Omega= 2\pi c_1(M,J)$ if and only if there exist non-trivial $\lambda_1$-eigenfunctions $f_1, \ldots, f_{\ell}$ such that the zero mean-value part of the sum $\sum_{i=1}^\ell f_i^2$ is a (possibly trivial) $\lambda_1$-eigenfunction, that is
\begin{equation}
\label{criterion}
\sum_{i=1}^\ell f_i^2-\frac{1}{{\rm vol}(M,g)}\int_M(\sum_{i=1}^\ell f_i^2)v_g\in E_1(g).
\end{equation}

\end{thm}
\begin{proof} 
We show that the necessary and sufficient condition in Theorem~\ref{l:extremal} is equivalent to relation~\eqref{criterion}. First, note that by Proposition~\ref{thm:matsushima}, we have $\lambda_1(g)=2$, and any eigenfunction $f \in E_1(g)$ is a Killing potential. In particular, the form $d^cf$ is dual to a Killing vector field, and therefore $D d^cf =(1/2)dd^cf$. Thus, using the standard identity $D^*D\ |T|^2= 2(D^*DT, \ T) - 2(DT, DT)$ for a  tensor field $T$,
we obtain
\begin{multline*}
\Delta_g\abs{df}^2= D^*D(d^cf, d^cf) = 2\Big((D^*D (d^cf), d^cf)- (D(d^cf), D (d^c f))\Big)\\ 
=(\delta d \ d^c f, d^cf)-(dd^c f, dd^cf )= (\Delta_g d^cf, d^c f)- (dd^cf, dd^c f)\\
= 2\abs{df}^2-\abs{dd^c f}^2. \hspace{2cm}
\end{multline*}
Above we also used  the relation $\Delta_g df^c = 2 df^c$ and the fact that the tensor norm of $dd^cf$ is twice its norm as a differential $2$-form. On the other hand, we clearly have
\begin{equation}
\label{Delta2}
\Delta_g f^2=2f \Delta_g f - 2\abs{df}^2=4f^2 - 2\abs{df}^2.
\end{equation}
Combining the last two relations, for any eigenfunction $f\in E_1(g)$ we obtain 
\begin{equation}
\label{Delta3}
\Delta_g \Big(f^2 -\abs{df}^2\Big) = 4f^2 -4\abs{df}^2 +\abs{dd^c f}^2.
\end{equation}
Now comparing~\eqref{Delta3} with the relation $\sum L(f_i)=0$ in Theorem~\ref{l:extremal}, we conclude that 
the metric $g$ is $\lambda_1$-extremal if and only if there exist $\{f_1, \ldots, f_{\ell}\} \in E_1(g)$ such that
\begin{equation}\label{extremal1}
\sum_{i=1}^{\ell}(f_i^2 -\abs{df_i}^2)= c  
\end{equation}                                  
for some constant $c$.  Integrating the last relation, we see that the constant $c$ is minus the mean-value of the sum $\sum_{i=1}^{\ell} f_i^2$. Setting  $f_0:=\sum_{i=1}^{\ell} f_i^2 + c$, and using~\eqref{Delta2}, we further obtain
\begin{equation}\label{extremal2}
\begin{split}
\Delta_g f_0 &=  \sum_{i=1}^{\ell}\Big(4 f_i^2 - 2\abs{df_i}^2\Big) \\
                    &= 2 f_0 + 2 \Big(\sum_{i=1}^{\ell}( f_i^2 -\abs{df_i}^2) - c\Big)=2f_0.
\end{split}
\end{equation}
Thus, relation~\eqref{extremal1} is in turn equivalent to the hypothesis that there exist non-trivial $f_1, \ldots, f_{\ell} \in E_1(g)$ such that the zero mean-value part $f_0$ of the sum $\sum_{i=1}^{\ell} f_i^2$ is a $\lambda_1$-eigenfunction itself. 
\end{proof}

Theorem~\ref{thm:main} is a useful criterion for verifying whether the K\"ahler--Einstein metric on $(M,J)$ is $\lambda_1$-extremal in $2\pi c_1(M,J)$. We demonstrate this in the corollaries below. 
\begin{cor}\label{c:homogeneous} Let $(M,g,J,\omega)$ be a compact homogeneous K\"ahler--Einstein manifold. Then the metric $g$ is a $\lambda_1$-extremal metric within its K\"ahler class.
\end{cor}
\begin{proof} 
Let $\{ f_1, \ldots, f_{\ell}\}$ be an orthonormal basis of $E_1(g)$ with respect to the  $L_2$-global product $\langle \cdot, \cdot \rangle_g$ induced by $g$.  The group of K\"ahler isometries $G$ of $(M,g,J, \omega)$ acts isometrically on the space $(E_1(g), \langle \cdot, \cdot \rangle_g)$. It follows that the function
$$f= \sum_{i=1}^{\ell} f_i^2$$ is $G$-invariant, and since $G$ acts transitively on $M$, is constant. 
\end{proof}

As another application, we consider  {\it toric} K\"ahler--Einstein manifolds which have been studied in many places, see~\cite{Abreu0, Do2, Guillemin, WZ}, and for which the existence theory takes a fairly concrete shape.
\begin{cor}
\label{c:toric} 
Let $(M,g,J, \omega)$ be a compact K\"ahler--Einstein manifold  of real dimension $2m$ whose connected identity component ${\rm Aut}_0(M,J)$ of the automorphism group is the  complexification  of an $m$-dimensional real torus. Then the metric $g$ is not $\lambda_1$-extremal.
\end{cor}
\begin{proof}
The assumption on ${\rm Aut}_0(M,J)$ implies that $(M,g,J,\omega)$ is a {\it toric} K\"ahler--Einstein metric in the sense of  \cite{Abreu0, Do2, Guillemin, WZ}. Indeed, by Proposition~\ref{thm:matsushima}, the connected component of the  isometry group of the K\"ahler--Einstein  metric is  a maximal connected compact subgroup of ${\rm Aut}_0(M,J)$. In our case, by assumption,  it must be a real $m$-dimensional torus $T$. The latter acts in a hamiltonian way (as any induced vector field is individually hamiltonian because $M$ is simply-connected, and $T$ is abelian), and by Delzant theorem \cite{Delzant}, the momentum map $\mu: M \to \mathfrak{t}^*$ sends $M$ onto a compact convex polytope in the dual vector space $\mathfrak{t}^*$ of the Lie algebra $\mathfrak{t} = {\rm Lie}(T)$. By Proposition~\ref{thm:matsushima} the pullback $f=(u, \mu) + \lambda$ to $M$ of an affine function $(u, x) + \lambda$ on $\mathfrak{t}^*$, where $u\in \mathfrak{t}$ and $\lambda\in \R$, defines an element in $E_1(g)\oplus \R$. Conversely, all elements of $E_1(g)\oplus \R$ are of this form, again  by Proposition~\ref{thm:matsushima} and our assumption for ${\rm Aut}_0(M,J)$. 

Suppose that the metric $g$ is $\lambda_1$-extremal. Then, by Theorem~\ref{thm:main}, there exist non-trivial eigenfunctions $f_{i}=(u_i, \mu) + \lambda_i$ such that the sum $\sum_{i=1}^{\ell} f_{i}^2$ is the pull-back  $f=(u,\mu) + \lambda$ of an affine function $(u,x) + \lambda$ on $\mathfrak{t}^*$. It follows that 
$$\sum_{i=1}^{\ell} \Big((u_i, x)+ \lambda_i\Big)^2 = (u,x) + \lambda,$$
which implies  $u_i=0$ for any $i$, and hence, $f_{i}=\lambda_i$ is constant for any $i$. Thus, we arrive at a contradiction with the hypothesis that the $f_i$'s are non-trivial.
\end{proof}
The above corollary, for instance,  shows that the K\"ahler--Einstein metric on ${\mathbb C}P^2$ blown-up at three points  in general position (see \cite{siu,T0}) is not $\lambda_1$-extremal, and in particular, can not  be a maximizer for $\lambda_1$ in its K\"ahler class.

\vspace{0.2cm}
As a final example, we consider non-homogeneous K\"ahler--Einstein metrics on projective bundles over the product of compact K\"ahler--Einstein manifolds with positive scalar curvatures that have been found and studied by Koiso and Sakane~\cite{Koi90, KS86, KS88, Sak}, see also \cite{hfkg4,DW,PS} for alternative treatments. It is convenient to use the characterization from \cite{hfkg4} saying that these are K\"ahler--Einstein metrics admitting a {\it Hamiltonian} $2$-form of order $1$,  in the sense of the theory in \cite{hfkg1,hfkg2}.  
We can then prove the following statement.
\begin{cor}
\label{c:hamiltonian-form} 
Let $(M, g, J, \omega)$ be a compact K\"ahler--Einstein manifold that is different from $\C P^m$ and  which admits a Hamiltonian $2$-form of order $\geqslant 1$. Then the metric $g$ is not $\lambda_1$-extremal.
\end{cor}
\begin{proof} 
As follows from the theory in \cite{hfkg1,hfkg2}, a Hamiltonian $2$-form $\phi$ on $(M,g, J, \omega)$ gives rise to an $\ell$-dimensional (real) torus $T$ in the connected component  $I_0(M,g)$ of the isometry group of $(g, J, \omega)$, where $\ell \geqslant 1$ is the order of $\phi$.  By the general classification~\cite[Thm.~5]{hfkg2} and \cite[Prop.~16]{hfkg1}, the corresponding moment map $\mu : M \to \mathfrak{t}^*$ sends $M$ to a Delzant simplex in the dual vector space of $\mathfrak{t}= {\rm Lie}(T)$. Since  $M$ is simply connected by the K\"ahler--Einstein assumption (see \cite[Sec.~2.1]{hfkg4} for the refinement in this case) and $(M,J) \ncong\C P^m$,  we conclude that $(M,J)$ is the total space of a holomorphic projective bundle $P({\mathcal E}_0 \oplus {\mathcal E}_1 \oplus \cdots \oplus {\mathcal E}_{\ell}) \to S$  over the product $S=\prod_j S_j$ of K\"ahler--Einstein manifolds $(S_j, g_j, J_j, \omega_j)$,  where ${\mathcal E}_i$ are projectively-flat holomorphic vector bundles over $S$,  satisfying certain topological conditions.  Moreover, the metric $g$ on $M$ is obtained by the {\it generalized Calabi construction}  associated to this bundle, see \cite{hfkg5} for a detailed treatment of this class of metrics, but the existence of a Hamiltonian $2$-form of order $\ell\ge 2$ is a more restrictive condition for the metric (and the bundles).  

To describe the eigenspace $E_1$ corresponding to $\lambda_1(g)$, we first show that the torus $T$ lies in the centre of $I_0(M,g)$. By the general theory in~\cite{hfkg1,hfkg2}, the torus $T$ is generated by the Hamiltonian vector fields corresponding to the elementary symmetric functions of the $m$-eigenvalues of the Hamiltonian form $\phi$ viewed as a Hermitian operator via the K\"ahler  form $\omega$.  Thus, it is sufficient to show that $\phi$ is invariant under any  isometry $\Phi \in I_0(M,g)$. As shown in \cite{FKMR}, when $(M,J)\ncong \C P^m$ any other Hamiltonian $2$-form $\tilde \phi$ on $(M, g, J, \omega)$ must be a linear combination of $\phi$ and $\omega$.  Since the property of being Hamiltonian is preserved by K\"ahler isometries, we see that $\tilde \phi =\Phi^*(\phi) = a\omega + b \phi$ for some constants $a,b$. 
In addition, we also have 
$$
{\rm tr}_{\omega}(\phi)= {\rm tr}_{\omega} \tilde \phi=a m+b{\rm tr}_{\omega}\phi.
$$
As a Hamiltonian $2$-form of order $\ell \ge 1$, $\phi$ cannot have a constant trace (otherwise it must be parallel by its very definition, and thus of order $\ell=0$, see \cite{hfkg1}), and we obtain that $a=0$ and $b=1$; in other words, $\Phi^*(\phi)=\phi$.

Denote by $\mathfrak{i}(M,g)$ the Lie algebra of $I_0(M,g)$. Since $M$ is Fano, it is simply connected, and we may identify $\mathfrak{i}(M,g)$ with the space of zero mean-value Killing potentials. It follows from our previous argument that the centralizer of $\mathfrak{t}$ in $\mathfrak{i}(M,g)$ equals $\mathfrak{i}(M,g)$, and we may use the description of $\mathfrak{i}(M,g)$ in \cite[Lemma~5]{hfkg5} in terms of the pullback metric on the $T$-equivariant blow-up
$$\hat M= P\Big(\bigoplus_{k=0}^{\ell} \cO(-1)_{\mathcal{E}_k}\Big) \to \hat S$$
of $(M,J)$ along  the sub-manifolds $P(\mathcal{E}_i)\cong \C P^{r_i-1}\times S$, where $r_i$ is the rank of  $\mathcal{E}_i$ and
$$\hat S= P(\mathcal{E}_0)\times_S \cdots \times_S P(\mathcal{E}_{\ell}) \cong \Big(\prod_{i=0}^{\ell} \C P^{r_i-1}\Big)\times \Big(\prod_{j}S_j\Big)$$
is equipped with the product of the K\"ahler--Einstein metrics on its factors. (Above we use the facts that $S$ is simply connected and the bundles $\mathcal{E}_i$ are projectively flat.) Note also that the induced action of $T$ on  $\hat M$ arises from the diagonal action on each fibre of $\Big(\bigoplus_{k=0}^{\ell} \cO(-1)_{\mathcal{E}_k}\Big) \to \hat S$, so that $\hat M$ has a structure of  a toric $\mathbb{C}P^{\ell}$-bundle over $\hat S$, such that the pull-back of $\omega$ to $\hat M$ restricts to each fibre to define a $T$-invariant symplectic form on $\mathbb{C}P^{\ell}$ whose momentum map is  the pull-back of $\mu$ to $\hat M$.

Denote by $(S_\alpha,g_\alpha, J_\alpha, \omega_\alpha)$ a K\"ahler-Einstein factor in the definition of $\hat S$ , where $\alpha$ ranges over the values of the $i$'s and $j$'s. As is shown in the proof of \cite[Lemma~5]{hfkg5}, any Killing potential $f$ on $(M,g, J, \omega)$  when pulled-back to $\hat M$ has the form
$$f = \sum_{\alpha} \Big((u_{\alpha}, \mu) + \lambda_{\alpha}\Big)f_{\alpha}  + \Big((u, \mu) + \lambda\Big),$$
where $f_\alpha$ is the pull-back to $\hat M$ of a Killing potential on $S_{\alpha}$, $u, u_{\alpha} \in \mathfrak{t}$, and $\lambda, \lambda_{\alpha} \in \R$. Here the values $u_{\alpha}, \lambda_{\alpha}$ are determined by $(M,J)$. More precisely,  the $u_{\alpha}$'s can be expressed in terms of  the degrees of ${\mathcal E}_i$ over $S_{j}$ while the $\lambda_{\alpha}$'s are determined by the K\"ahler class of $\omega$.

By Theorem~\ref{thm:main}, the hypothesis that  the metric $g$  is $\lambda_1$-extremal reduces to the relation
$$
\sum_{n=1}^{r} \Big(\sum_{\alpha} \Big((u_{\alpha}, \mu) + \lambda_{\alpha}\Big)f_{\alpha,n}  + (u_n, \mu) + \lambda_n\Big)^2= \sum_{\alpha} \Big((u_{\alpha}, \mu) + \lambda_{\alpha}\Big)f_{\alpha}  + (u, \mu) + \lambda,
$$
where $r\geqslant 2$, and  the functions
$$f_n:=\Big(\sum_{\alpha} \Big((u_{\alpha}, \mu) + \lambda_{\alpha}\Big)f_{\alpha,n}  + (u_n, \mu) + \lambda_n\Big)$$
are non-constant. Restricting to a $\mathbb{C} P^{\ell}$ fibre over $\hat S$, the argument from the proof of Corollary~\ref{c:toric} shows that  the following relations hold:
\begin{equation}\label{cond}
\sum_{\alpha} u_{\alpha} f_{\alpha,n} + u_n=0, \qquad\sum_{\alpha} u_{\alpha} f_{\alpha} + u=0,
\end{equation}
where $n=1,\ldots, r$ is arbitrary. Here $f_{\alpha,n}$ and $f_{\beta, n}$ are the pullbacks of Killing potentials on $S_\alpha$ and $S_\beta$ respectively, and hence, are functionally independent  when $\alpha \neq \beta$. Thus, we see that $u_{\alpha} f_{\alpha,n}$ and $u_{\alpha} f_{\alpha}$ are constant vectors in the Lie algebra $\mathfrak{t}$ for any $\alpha$. By the classification in \cite[Thm.~5]{hfkg2}, we get  that $u_{\alpha} \neq 0$; relation~$(40)$ in \cite{hfkg2} cannot vanish for each $j$. Hence, the potentials $f_{\alpha,n}$ and $f_{\alpha}$ are constant, and by the relations in~\eqref{cond}, so are the functions
$$f_n=\Big(\sum_{\alpha} \Big((u_{\alpha}, \mu) + \lambda_{\alpha}\Big)f_{\alpha,n}  + (u_n, \mu) + \lambda_n\Big)= \sum_{\alpha} \Big(\lambda_{\alpha}f_{\alpha,n}\Big)  + \lambda_n.$$
Thus, we arrive at a contradiction. \end{proof}

\subsection*{Acknowledgements} 
The project has originated out of a number of discussions the authors had while GK was visiting the Centre de recherches math\'ematiques in Montr\'eal. Its hospitality is gratefully acknowledged. During the work on the paper VA and DJ were supported by NSERC Discovery Grants; DJ was also partially supported by FQRNT. The authors are grateful to the referee for carefully reading the manuscript and suggesting a number of improvements.

\end{document}